\documentclass[reqno,12pt]{amsart}
\usepackage{amsfonts}
\usepackage{bbm}
\usepackage{}
\numberwithin{equation}{section}
\usepackage{color}
\usepackage{rotating}
\usepackage{indentfirst}
\usepackage{color}
\usepackage{amssymb}
\usepackage{mathrsfs}
\usepackage{xy}
\usepackage{rotating}
\usepackage{tikz}

\xyoption{all}
\def\Ext{\mbox{\rm Ext}} \def\Hom{\mbox{\rm Hom}} \def\dim{\mbox{\rm dim}\,} 
\def\lr#1{\langle #1\rangle}    \def\mod{\mbox{\rm \textbf{mod}}\,}
 
\def\End{\mbox{\rm End}\,}
\def\ind{\mbox{\rm ind}\,}\def\add{\mbox{\rm add}\,}
\def\A{\mathcal{A}\,} 
\def\Cone{\mbox{\rm Cone}\,} \def\Hom{\mbox{\rm Hom}\,}  \def\Im{\mbox{\rm Im}\,}

\theoremstyle{plain}
\newtheorem{theorem}{\bf Theorem}[section]
\newtheorem{lemma}[theorem]{\bf Lemma}
\newtheorem{corollary}[theorem]{\bf Corollary}
\newtheorem{proposition}[theorem]{\bf Proposition}

\theoremstyle{definition}
\newtheorem{definition}[theorem]{\bf Definition}
\newtheorem{remark}[theorem]{\bf Remark}
\newtheorem{example}[theorem]{\bf Example}

\newtheorem{condition}[theorem]{\bf Condition}

\newcommand{\bt}{\begin{theorem}}
\newcommand{\et}{\end{theorem}}
\newcommand{\bl}{\begin{lemma}}
\newcommand{\el}{\end{lemma}}
\newcommand{\bd}{\begin{definition}}
\newcommand{\ed}{\end{definition}}
\newcommand{\bc}{\begin{corollary}}
\newcommand{\ec}{\end{corollary}}
\newcommand{\bp}{\begin{proof}}
\newcommand{\ep}{\end{proof}}
\newcommand{\bx}{\begin{example}}
\newcommand{\ex}{\end{example}}
\newcommand{\br}{\begin{remark}}
\newcommand{\er}{\end{remark}}
\newcommand{\be}{\begin{equation}}
\newcommand{\ee}{\end{equation}}
\newcommand{\ba}{\begin{align}}
\newcommand{\ea}{\end{align}}
\newcommand{\bn}{\begin{enumerate}}
\newcommand{\en}{\end{enumerate}}
\newcommand{\bcs}{\begin{cases}}
\newcommand{\ecs}{\end{cases}}
\newcommand{\RNum}[1]{\uppercase\expandafter{\romannumeral #1\relax}}

\makeatletter
\renewcommand{\section}{\@startsection{section}{1}{0mm}
  {-\baselineskip}{0.5\baselineskip}{\bf\leftline}}
\makeatother

\begin{document}
\title[Cluster characters for Frobenius extriangulated categories]{Cluster characters for 2-Calabi-Yau \\Frobenius extriangulated categories}
\author[L. Wang, J. Wei, H. Zhang]{Li Wang, Jiaqun Wei,  Haicheng Zhang}
\address{School of Mathematics-Physics and Finance, Anhui Polytechnic University, Wuhu 241000, P. R. China.\endgraf}
\email{wl04221995@163.com (Wang)}
\address{Institute of Mathematics, School of Mathematical Sciences, Nanjing Normal University,
 Nanjing 210023, P. R. China.\endgraf}
\email{weijiaqun@njnu.edu.cn (Wei)}
\address{Institute of Mathematics, School of Mathematical Sciences, Nanjing Normal University,
 Nanjing 210023, P. R. China.\endgraf}
\email{zhanghc@njnu.edu.cn (Zhang)}


\begin{abstract}
We define the cluster characters for 2-Calabi-Yau Frobenius extriangulated categories with cluster tilting objects. This provides a unified framework of cluster characters in 2-Calabi-Yau triangulated categories and 2-Calabi-Yau Frobenius exact categories given by Palu and Fu--Keller, respectively.
\end{abstract}

\maketitle

\section{Introduction}
Cluster algebras were introduced by Fomin and Zelevinsky in \cite{Fo}, and  since then played an important role in Lie theory and  representation theory. They are commutative algebras generated by the cluster variables which are produced by mutating seeds. The  relationship between cluster algebra and tilting theory was disclosed in \cite{Ma}, which was also the main motivation to introduce the cluster category (cf. \cite{Bu2}). Cluster categories are a class of $2$-CY ($2$-Calabi-Yau) triangulated categories admitting cluster tilting objects and can be used to ``categorify'' cluster algebras. To this end, one needs to define a map from the considered category $\mathscr{C}$ to the corresponding Laurent polynomial algebra, called the cluster character on $\mathscr{C}$, such that it satisfies some properties and realizes the exchange relations in the cluster algebra.

Let $\mathscr{C}$ be a $2$-CY triangulated category with a cluster tilting object $T$. Palu \cite{Pa} defined the index with respect to $T$ and showed that the Caldero--Chapoton map introduced in [2, Section 6] is a cluster character on $\mathscr{C}$. Fu--Keller \cite{Chan} defined a cluster character on $2$-CY Frobenius exact categories and established the link between cluster algebras with coefficients and $2$-CY Frobenius exact categories via cluster characters.

Recently, Nakaoka--Palu \cite{Na} introduced a notion of extriangulated categories by extracting properties on triangulated categories and exact categories. The extriangulated categories have gotten significant applications in the categorifications of some related notions in cluster algebras in which the indices, Grothendieck groups and $n$-cluster tilting subcategories play a crucial role. For instances, Padrol--Palu--Pilaud--Plamondon \cite{PPPP} studied the relations between $g$-vectors of finite type cluster algebras via the indices, Grothendieck groups in extriangulated categories. Wang--Wei--Zhang \cite{Wl} developed a general framework for $c$-vectors in cluster algebras using the indices in 2-CY extriangulated categories. Recently, Wu \cite{Wu} generalized the construction of (generalized) cluster categories and introduced the relative cluster categories. Then he considered an extension closed subcategory of the relative cluster category, which was proved to be a 2-CY Frobenius extriangulated category with a cluster tilting object, called the Higgs category. Keller--Wu \cite{Kellerwu} used the Higgs category to provide the categorification of cluster algebras with coefficients. In particular, they defined a canonical cluster character on the Higgs category.

The main aim of this paper is to provide a unified framework of cluster characters in 2-CY triangulated categories and 2-CY Frobenius exact categories by defining the cluster characters for 2-CY Frobenius extriangulated categories.

The paper is organized as follows: we summarize some basic definitions and properties of extriangulated categories and indices in Section 2. In Section 3, we investigate the indices with respect to $n$-cluster tilting subcategories and related Grothendieck groups in Frobenius extriangulated categories. In particular, for an $n$-cluster tilting object $T$, we define a map $\Phi: K_{0}(\mod C)\rightarrow K_{0}^{\rm sp}(\add{T})$ (see Theorem \ref{main1}). Section 4 is devoted to defining the cluster characters of 2-CY Frobenius extriangulated categories $\mathscr{C}$ by using the map $\Phi$ and proving the multiplication formulas of cluster characters. Moreover, we compare the defined cluster characters with those given by Palu \cite{Pa} and Fu--Keller \cite{Chan}. At the end, as an example, we take $\mathscr{C}$ to be the category of finite-dimensional modules over the preprojective algebra of a Dynkin quiver of type $\mathbb{A}_{2}$ and compare our cluster characters with those given by Fu--Keller.

Recall that a category is called {\em Krull--Schmidt} if any object can be written as a finite direct sum of objects whose endomorphism rings are local. Throughout the paper, let $k$ be an algebraically closed field. We assume that all considered categories are skeletally small, $k$-linear and Krull--Schmidt, and the subcategories are full and closed under isomorphisms. Let $\mathscr{C}$ be an additive category, for an object $X\in\mathscr{C}$, we denote by $\End_{\mathscr{C}}(X)$ and $[X]$  the endomorphism ring and the isomorphism class of $X$, respectively, and denote by $\add X$ the full subcategory of $\mathscr{C}$ formed by the finite direct sums of direct summands of $X$. For a $k$-algebra $\Lambda$, we denote by $\mod \Lambda$ the category of finitely generated $\Lambda$-modules.

\section{Preliminaries}
In this section, let us recall some notions and properties of extriangulated categories and indices from \cite{Na} and \cite{Wl}.

\subsection{Extriangulated categories}
We briefly recall the definition and some basic properties of extriangulated categories. For more details on extriangulated categories we refer to \cite[Section 2]{Na}.

An {\em extriangulated category} $(\mathscr{C}, \mathbb{E},\mathfrak{s})$ is a triple consisting of the following data satisfying certain axioms:

(1) $\mathscr{C}$ is an additive category and $\mathbb{E}:\mathscr{C}^{\rm op}\times \mathscr{C}\rightarrow {\bf Ab}$ is a bifunctor.

(2)  $\mathfrak{s}$ is an additive realization of $\mathbb{E}$, which  defines the class of conflations satisfying the axioms
(ET1)-(ET4), (ET3)$^{\rm op}$ and (ET4)$^{\rm op}$ (see \cite[Definition 2.12]{Na}).

Exact categories, triangulated categories are extriangulated categories (cf. \cite[Example 2.13]{Na}). Given $\delta\in\mathbb{E}(C,A)$, if $\mathfrak{s}(\delta)=[A\stackrel{x}{\longrightarrow}B\stackrel{y}{\longrightarrow}C]$, then the sequence $A\stackrel{x}{\longrightarrow}B\stackrel{y}{\longrightarrow}C$ is called a {\em conflation}, $x$ is called an {\em inflation} and $y$ is called a {\em deflation}. In this case, we call $$A\stackrel{x}{\longrightarrow}B\stackrel{y}{\longrightarrow}C\stackrel{\delta}\dashrightarrow$$
is an $\mathbb{E}$-triangle. An object $P$ in $\mathscr{C}$ is called {\em projective} if $\mathbb{E}(P,\mathscr{C})=0$. Denote by $\mathcal{P}$ the full subcategory of projective objects in $\mathscr{C}$. We say that $\mathscr{C}$ {\em has enough projective objects} if for any object $M\in\mathscr{C}$, there exists an $\mathbb{E}$-triangle $$A\stackrel{}{\longrightarrow}P\stackrel{}{\longrightarrow}M\stackrel{}\dashrightarrow$$
such that $P\in\mathcal{P}$. Dually, we define {\em injective objects} and that $\mathscr{C}$ {\em has enough injective objects}. According to \cite{Zhou}, the extriangulated category $\mathscr{C}$ is called {\em 2-Calabi-Yau}, if there exists a bifunctorial isomorphism $$\mathbb{E}(N,M)\cong {\rm D}\mathbb{E}(M,N)$$ for any $M,N\in\mathscr{C}$, where ${\rm D}=\Hom_k(-,k)$ is the usual $k$-linear duality.

Given two subcategories $\mathcal{T},\mathcal{F}$ of $\mathscr{C}$, set
$$\mathcal{T}\ast \mathcal{F}=\{M\in \mathscr{C}~|~\text{there is an $\mathbb{E}$-triangle}~T\stackrel{}{\longrightarrow}M\stackrel{}{\longrightarrow}F\stackrel{}\dashrightarrow~\text{with}~T\in\mathcal{T},F\in\mathcal{F} \}$$
and
$$\Cone(\mathcal{T},\mathcal{F})=\{M\in \mathscr{C}~|~\text{there is an $\mathbb{E}$-triangle}~T\stackrel{}{\longrightarrow}F\stackrel{}{\longrightarrow}M\stackrel{}\dashrightarrow\text{with}~T\in\mathcal{T},F\in\mathcal{F} \}.$$
Until the end of the paper, we always assume that $\mathscr{C}$ is an extriangulated category, unless otherwise stated.

\subsection{Indices and opposite indices}
Let $\mathcal{T}$ be a subcategory of $\mathscr{C}$. A morphism $f:X\rightarrow T$ with $T\in\mathcal{T} $ is called a {\em left $\mathcal{T}$-approximation} of $X$ if $\Hom_{\mathscr{C}}(f,\mathcal{T})$ is an epimorphism. The subcategory $\mathcal{T}$ is called {\em covariantly finite} if every object in $\mathscr{C}$ admits a left $\mathcal{T}$-approximation. Dually, we can define the notions of {\em right  $\mathcal{T}$-approximations} and {\em contravariantly finite} subcategories. We call a subcategory {\em functorially finite} if it is both covariantly finite and contravariantly finite.

For each positive integer $n$, Gorsky--Nakaoka--Palu \cite{Go} defined the higher positive extensions $\mathbb{E}^n(-,-)$ in an extriangulated category. If $\mathscr{C}$ has enough projective objects or enough injective objects, these $\mathbb{E}^n$ are isomorphic to those defined in \cite{LN} (cf. \cite[Remark 3.4]{Go}). For a collection $\mathcal{X}$ of objects in $\mathscr{C}$, define the full subcategory
$$^{\perp_{[1,n]} }\mathcal{X}=\{M\in \mathscr{C}~|~\mathbb{E}^{i}(M,\mathcal{X})=0~\text{for}~i=1,2\cdots n \}.$$
Dually, one defines the full subcategory $\mathcal{X}^{\perp_{[1,n]} }$ in $\mathscr{C}$.
\begin{definition} (\cite[Definition 3.1]{Wl}) A {\em tower} of triangles in $\mathscr{C}$ is a diagram of the form
\begin{equation*}
\xymatrix{
                & T_{1}\ar[dr]^{g_{1}} \ar[rr]^{f_{1}g_{1}}      && T_{2}\ar[dr]^{g_{2}}      &\cdots&\cdots&\cdots&T_{n-1}\ar[dr]^{g_{n-1}} \\
  T_{0}\ar[ur]^{f_{0}}  & & K_{1}\ar@{-->}[ll]_{\delta_{0}}\ar[ur]^{f_{1}}     &&K_{2}\ar@{-->}[ll]_{\delta_{1}}     &\cdots&K_{n-2}\ar[ur]^{f_{n-2}}   &  &\ar@{-->}[ll]_{\delta_{n-2}} T_{n} }
\end{equation*}
where a dotted arrow $X\stackrel{\delta}\dashleftarrow Y$ means an extension $\delta\in\mathbb{E}(Y,X)$, each oriented triangle is an $\mathbb{E}$-triangle, and each non-oriented triangle is commutative.
\end{definition}
A subcategory $\mathcal{T}$ of $\mathscr{C}$ is called {\em n-cluster tilting} if
$\mathcal{T}$ is functorially finite and $\mathcal{T}=\mathcal{T}^{\perp_{[1,n-1]}}={^{\perp_{[1,n-1]}}}\mathcal{T}$. An object $T$ in $\mathscr{C}$ is called {\em n-cluster tilting} if $\add T$ is an $n$-cluster tilting subcategory of $\mathscr{C}$.

Let $G(\mathscr{C})$ be the free abelian group spanned by the isomorphism classes of objects in $\mathscr{C}$. Define the {\em split Grothendieck group} of $\mathscr{C}$  to be
$$K_{0}^{\rm sp}(\mathscr{C})=G(\mathscr{C})/\lr{[A\oplus B]-[A]-[B]~|~A,B\in\mathscr{C}}$$
and the {\em Grothendieck group} of $\mathscr{C}$ to be
$$K_{0}(\mathscr{C})=K_{0}^{\rm sp}(\mathscr{C})/\lr{[A]-[B]+[C]~|~A\rightarrow B\rightarrow C\dashrightarrow~\text{is an $\mathbb{E}$-triangle}}.$$

From now on, we always assume that $\mathscr{C}$ satisfies the following conditions (see \cite[Condition 5.8]{Na}).

\begin{condition}\label{WIC} (WIC)
(1) Let $f:X\rightarrow Y$ and  $g:Y\rightarrow Z$ be any pair of composable morphisms in $\mathscr{C}$. If $gf$ is an inflation, then $f$ is an inflation.

(2) Let $f:X\rightarrow Y$ and  $g:Y\rightarrow Z$ be any pair of composable morphisms in $\mathscr{C}$. If $gf$ is a deflation, then $g$ is a deflation.
\end{condition}
If $\mathscr{C}$ is a triangulated category or weakly idempotent complete exact category, then Condition \ref{WIC} is satisfied.

Given an $n$-cluster tilting subcategory $\mathcal{T}$ of $\mathscr{C}$. By \cite[Lemma 3.4]{Wl}, for any $M\in\mathscr{C}$, we have towers
\begin{equation*}
\xymatrix{
                & T_{n-2}\ar[dr]^{f_{n-2}} \ar[rr]^{}      && T_{n-3}\ar[dr]^{f_{n-3}}      &\cdots&\cdots&\cdots&T_{0}\ar[dr]^{f_{0}} \\
  T_{n-1}\ar[ur]^{}  & & K_{n-2}\ar@{-->}[ll]^{}\ar[ur]^{}     &&K_{n-3}\ar@{-->}[ll]^{}     &\cdots&K_{1}\ar[ur]^{}   &  &\ar@{-->}[ll]^{} M }
\end{equation*}
and
\begin{equation*}
\xymatrix{
                & T_{0}'\ar[dr]^{} \ar[rr]^{}      && T_{1}'\ar[dr]^{}      &\cdots&\cdots&\cdots&T_{n-2}'\ar[dr]^{} \\
 M\ar[ur]^{g_{0}}  & & K'_{1}\ar@{-->}[ll]^{}\ar[ur]^{g_{1}}     &&K'_{2}\ar@{-->}[ll]^{}     &\cdots&K'_{n-2}\ar[ur]^{g_{n-2}}   &  &\ar@{-->}[ll]^{} T_{n-1}', }
\end{equation*}
where each $T_{i},T'_{i}\in \mathcal{T}$, $K_{i}\in\mathcal{T}^{\perp_{[1,i]}}$ and  $K'_{i}\in{^{\perp_{[1,i]}}}\mathcal{T}$. In particular, $f_{i}$ is a minimal right $\mathcal{T}$-approximation and $g_{i}$ is a minimal left $\mathcal{T}$-approximation. The {\em index} of $M$ with respect to $\mathcal{T}$ is the element
$$\ind_{\mathcal{T}}(M)=\sum_{i=0}^{n-1}(-1)^{i}[T_{i}]\in K_{0}^{\rm sp}(\mathcal{T}).$$
Similarly, the {\em opposite index} of $M$ with respect to $\mathcal{T}$ is the element
$$\ind^{\rm op}_{\mathcal{T}}(M)=\sum_{i=0}^{n-1}(-1)^{i}[T'_{i}]\in K_{0}^{\rm sp}(\mathcal{T}).$$

The following results from \cite[Lemma 3.8]{Wl} will be frequently used in Section 3.
\begin{proposition}\label{useful} Let $\mathcal{T}$ be an $n$-cluster tilting subcategory of $\mathscr{C}$.

$(1)$ For any $M,N\in\mathscr{C}$, we have
$$\ind_{\mathcal{T}}(M\oplus N)=\ind_{\mathcal{T}}(M)+\ind_{\mathcal{T}}(N).$$

$(2)$ Given an $\mathbb{E}$-triangle \begin{equation*}
X\stackrel{f}{\longrightarrow}L\stackrel{g}{\longrightarrow}M\stackrel{}\dashrightarrow.
\end{equation*}

\begin{itemize}
\item [(i)] If $g$ is a right $\mathcal{T}$-approximation, then
$$ [L]=\ind_{\mathcal{T}}(X)+\ind_{\mathcal{T}}(M).$$
\item [(ii)] If $f$ is a left $\mathcal{T}$-approximation, then
$$ [L]=\ind^{\rm op}_{\mathcal{T}}(X)+\ind^{\rm op}_{\mathcal{T}}(M).$$
\item [(iii)] If $X\in \mathcal{T}^{\perp_{1}}$, then
$$ \ind_{\mathcal{T}}(L)=\ind_{\mathcal{T}}(X)+\ind_{\mathcal{T}}(M).$$
If $M\in {^{\perp_{1}}}\mathcal{T}$, then
$$ \ind^{\rm op}_{\mathcal{T}}(L)=\ind^{\rm op}_{\mathcal{T}}(X)+\ind^{\rm op}_{\mathcal{T}}(M).$$
\end{itemize}
\end{proposition}

\section{Indices and Grothendieck groups }
In this section, we assume that $\mathscr{C}$ is a {\em Frobenius extriangulated category}, i.e. $\mathscr{C}$  has enough projectives, enough injectives and the projectives coincide with the injectives. The stable category $\mathscr{C}/\mathcal{P}$ is denoted by $\mathcal{A}$. By \cite[Corollory 7.4]{Na}, $\mathcal{A}$ is a triangulated category with the suspension functor denoted by $\Sigma$. Let $\pi: \mathscr{C}\rightarrow\mathcal{A}$ be the projection functor.

\begin{lemma}\label{L-31} The triangles in $\mathcal{A}$ are exactly induced by the $\mathbb{E}$-triangles in $\mathscr{C}$.
\end{lemma}
\begin{proof} Let
$$A\stackrel{}{\longrightarrow}B\stackrel{}{\longrightarrow}C\stackrel{}\dashrightarrow$$
be an $\mathbb{E}$-triangle in $\mathscr{C}$, then there exists a commutative diagram of conflations
\begin{equation*}
\xymatrix{
   A\ar[d]_{} \ar[r]^{} &  B\ar[d]_{} \ar[r]^{} &  C\ar@{=}[d] \\
   I(A)\ar[d]_{} \ar[r]^{} &  I(A)\oplus C \ar[d]_{} \ar[r]^{} & C \\
   \Sigma A\ar@{=}[r] &  \Sigma A&    }
\end{equation*}
with $I(A)\in\mathcal{P}$ such that
\begin{equation}\label{1-1}
A\stackrel{}{\longrightarrow}B\stackrel{}{\longrightarrow}C{\longrightarrow}\Sigma A
\end{equation}
is a triangle in $\mathcal{A}$. Conversely, given a triangle (\ref{1-1}) in $\mathcal{A}$,  we have a commutative diagram
\begin{equation*}
\xymatrix{
 &  A\ar[d]^{} \ar[r]^-{} & I(A)\ar[d]^{} \ar[r]^-{} & \Sigma A\ar@{=}[d]  &  \\
&  B \ar[r]^-{} & C \ar[r]^-{} & \Sigma A &  }
\end{equation*}
with $I(A)\in\mathcal{P}$. By \cite[Proposition 1.20]{LN}, we obtain an $\mathbb{E}$-triangle
$$A\stackrel{}{\longrightarrow}B\oplus I(A)\stackrel{}{\longrightarrow}C\stackrel{}\dashrightarrow$$
whose image under the  functor $\pi$ is isomorphic to the given triangle (\ref{1-1}).
\end{proof}
\begin{remark}\label{dimension} Given an object $N\in\mathscr{C}$. Since $\mathscr{C}$ is Frobenius, for each positive integer $i$, we have an $\mathbb{E}$-triangle
$$\Sigma^{i-1}N\stackrel{}{\longrightarrow}P_{i}\stackrel{}{\longrightarrow}\Sigma^{i}N\stackrel{}\dashrightarrow$$ with $P_i\in\mathcal{P}$. Thus, by \cite{LN},
$\mathbb{E}^{i}(M,N)\cong \mathbb{E}(M,\Sigma^{i-1}N)$ for any $M\in\mathscr{C}$.
\end{remark}

By the definition of $n$-cluster tilting subcategories, we have the following
\begin{lemma}\label{cluster} Let $\mathcal{T}$ be an $n$-cluster tilting subcategory of $\mathscr{C}$ and $M\in\mathscr{C}$.

$(1)$ For any $X, Y\in \mathcal{T}$ and $1\leq i\leq n-1$, we have $$\Hom_{\mathcal{A}}(X,\Sigma^{i}Y)=\Hom_{\mathcal{A}}(\Sigma^{i}X,Y)=0.$$

$(2)$ If $\Hom_{\mathcal{A}}(M,\Sigma^{i}N)=0$ for any $N\in \mathcal{T}$ and $1\leq i\leq n-1$, then $M\in {^{\perp_{[1,n-1]}}}\mathcal{T}$.

$(2')$ If $\Hom_{\mathcal{A}}(\Sigma^{i}M,N)=0$ for any $N\in \mathcal{T}$ and $1\leq i\leq n-1$, then $M\in \mathcal{T}{^{\perp_{[1,n-1]}}}$.

$(3)$
 $\mathcal{T}$ is an $n$-cluster tilting subcategory of $\mathcal{A}$.
\end{lemma}
\begin{proof}
$(1)$ Let $f\in\Hom_{\mathcal{A}}(X,\Sigma^{i}Y)$. By Lemma \ref{L-31}, we can lift the triangle in $\mathcal{A}$
\begin{equation*}
\Sigma^{i-1}Y\stackrel{}{\longrightarrow}H\stackrel{}{\longrightarrow}X\stackrel{f}{\longrightarrow}\Sigma^{i} Y
\end{equation*}
to an  $\mathbb{E}$-triangle in $\mathscr{C}$
\begin{equation*}
\Sigma^{i-1}Y\stackrel{}{\longrightarrow}H\oplus P\stackrel{}{\longrightarrow}X\stackrel{\delta}\dashrightarrow
\end{equation*}
for some $P\in \mathcal{P}$. By Remark \ref{dimension}, $\mathbb{E}(X,\Sigma^{i-1}Y)\cong \mathbb{E}^{i}(X,Y)=0$. It follows that $\delta=0$ and then $f=0$. Hence, $\Hom_{\mathcal{A}}(X,\Sigma^{i}Y)=0$. Similarly, we can prove $\Hom_{\mathcal{A}}(\Sigma^{i}X,Y)=0$.

(2) For any $N\in\mathcal{T}$, by Remark \ref{dimension}, we only need to show that  $\mathbb{E}(M,\Sigma^{i-1}N)=0$ for $1\leq i\leq n-1$. For any $\delta\in\mathbb{E}(M,\Sigma^{i-1}N)$, the $\mathbb{E}$-triangle
\begin{equation*}
\Sigma^{i-1}N\stackrel{f}{\longrightarrow}H\stackrel{g}{\longrightarrow}M\stackrel{\delta}\dashrightarrow
\end{equation*}
induces a triangle
\begin{equation*}
\Sigma^{i-1}N\stackrel{\pi(f)}{\longrightarrow}H\stackrel{\pi(g)}{\longrightarrow}M\stackrel{h}{\longrightarrow}\Sigma^{i} N
\end{equation*}
in $\mathcal{A}$. By hypothesis, we have $h=0$ and thus $\pi(f)$ is a section in $\mathcal{A}$.  It follows that there exists $s\in\Hom_{\mathscr{C}}(H,\Sigma^{i-1}N)$ such that $sf-1$ factors through some projective-injective object $P:$
$$\xymatrix{
   \Sigma^{i-1}N\ar[rr]^{sf-1} \ar[dr]_{t_{1}}
                &  &     \Sigma^{i-1}N \\
                &     P   \ar[ur]_{t_{2}}           }
$$
Since $P$ is injective, there exists $l\in\Hom_{\mathscr{C}}(H,P)$ such that $t_{1}=lf$. Then $(s-t_{2}l)f=1$ and thus $f$ is a section in $\mathscr{C}$. Hence, we obtain $\delta=0$. Similarly, we prove $(2')$.

(3) It is evident that $\mathcal{T}$ is functorially finite in $\mathcal{A}$.  Hence, the statement follows from  (1), (2) and $(2')$.
\end{proof}

In what follows, we always assume that $\mathcal{T}$ is an $n$-cluster tilting subcategory of $\mathscr{C}$.
Since $\mathscr{C}$ is Frobenius, for  any $X\in\mathscr{C}$, there exists an $\mathbb{E}$-triangle
$$X'\stackrel{}{\longrightarrow}P\stackrel{}{\longrightarrow}X\stackrel{}\dashrightarrow$$
with $P\in\mathcal{P}$. Define $\Theta(X)$ to be the class $\ind_{\mathcal{T}}(X')-[P]+\ind_{\mathcal{T}}(X)$ in $K_{0}^{\rm sp}(\mathcal{T})$. Suppose that we take another  $\mathbb{E}$-triangle
$$X''\stackrel{}{\longrightarrow}P'\stackrel{}{\longrightarrow}X\stackrel{}\dashrightarrow$$
with $P'\in\mathcal{P}$, then we have $X'\oplus P'\cong P\oplus X''$ by using \cite[Proposition 3.15]{Na}.  By Proposition \ref{useful}, $\ind_{\mathcal{T}}(X')-[P]=\ind_{\mathcal{T}}(X'')-[P']$. Thus, $\Theta(X)$ only depends on $X$.

The following proposition gives a characterization of $\Theta$.
\begin{proposition}\label{2333} For each $\mathbb{E}$-triangle with $P\in\mathcal{P}$
$$X'\stackrel{}{\longrightarrow}P\stackrel{}{\longrightarrow}X\stackrel{}\dashrightarrow,$$
we have $\Theta(X)=\ind_{\mathcal{T}}(X')-\ind^{\rm op}_{\mathcal{T}}(X')$. In particular, if $X\in\Cone(\mathcal{T},\mathcal{P})$, $\Theta(X)=0$.
\end{proposition}
\begin{proof}   By \cite[Lemma 3.4]{Wl} and Proposition \ref{useful}, there exists a tower
 \begin{equation*}
\xymatrix{
                & T_{0}\ar[dr]^{} \ar[rr]^{}      && T_{1}\ar[dr]^{}      &\cdots&\cdots&\cdots&T_{n-2}\ar[dr]^{} \\
 X'\ar[ur]^{}  & & K_{1}\ar@{-->}[ll]^{}\ar[ur]^{}     &&K_{2}\ar@{-->}[ll]^{}     &\cdots&K_{n-2}\ar[ur]^{}   &  &\ar@{-->}[ll]^{} T_{n-1}}
\end{equation*}
such that each $T_{i}\in \mathcal{T}$, $K_{i}\in\mathcal{T}^{\perp_{[1,i]}}$, $[T_{0}]=\ind^{\rm op}_{\mathcal{T}}(X')+\ind^{\rm op}_{\mathcal{T}}(K_{1})$ and $[T_{i}]=\ind^{\rm op}_{\mathcal{T}}(K_{i})+\ind^{\rm op}_{\mathcal{T}}(K_{i+1})$. Consider the following commutative diagram of conflations
\begin{equation*}
\xymatrix{
   X'\ar[d]_{} \ar[r]^{} &  P\ar[d]_{} \ar[r]^{} &  X\ar@{=}[d]^{} \\
   T_{0}\ar[d]_{} \ar[r]^{} & P\oplus K_{1} \ar[r]^{}\ar[d]_{} & X \\
  K_{1} \ar@{=}[r]^{} & K_{1}. &    }
\end{equation*}
By Proposition \ref{useful}, we get
$$\ind_{\mathcal{T}}(X)=[P]+\ind_{\mathcal{T}}(K_{1})-[T_{0}].$$
Thus, we obtain
\begin{align*}
\Theta(X)&=\ind_{\mathcal{T}}(X')-[P]+\ind_{\mathcal{T}}(X)\\
&=\ind_{\mathcal{T}}(X')-[T_{0}]+\ind_{\mathcal{T}}(K_{1})\\
&=\ind_{\mathcal{T}}(X')-\ind^{\rm op}_{\mathcal{T}}(X')+\ind_{\mathcal{T}}(K_{1})-\ind^{\rm op}_{\mathcal{T}}(K_{1}).
\end{align*}
Similarly, for $K_{2}$, there is an $\mathbb{E}$-triangle
$$K'_{2}\stackrel{}{\longrightarrow}T'_{0}\stackrel{}{\longrightarrow}K_{2}\stackrel{}\dashrightarrow$$
with $T'_{0}\in\mathcal{T}$, $K'_{2}\in \mathcal{T}^{\perp_{1}}$ and $[T'_{0}]=\ind_{\mathcal{T}}(K_{2})+\ind_{\mathcal{T}}(K'_{2})$.
Then we have the following commutative diagram
\begin{equation*}
\xymatrix{   &   & K'_{2}\ar[d]\ar@{=}[r] & K'_{2}\ar[d]^{} &\\
  & K_{1} \ar[r]\ar@{=}[d]  &  H \ar[d]_{} \ar[r]&  T'_{0} \ar[d] \\
   & K_{1}\ar[r]  &  T_{1} \ar[r]&  K_{2}\,.  & }
\end{equation*}
Since $K_{1}\in \mathcal{T}^{\perp_{1}}$ and $K'_{2}\in \mathcal{T}^{\perp_{1}}$, we get
$$H\cong K'_{2}\oplus T_{1}\cong K_{1}\oplus T'_{0}. $$
Hence, we have
\begin{align*}
\ind_{\mathcal{T}}(K_{1})-\ind^{\rm op}_{\mathcal{T}}(K_{1})&=\ind_{\mathcal{T}}(K_{1})-([T_{1}]-\ind^{\rm op}_{\mathcal{T}}(K_{2}))\\
&=(\ind_{\mathcal{T}}(K'_{2})+[T_{1}]-[T'_{0}])-([T_{1}]-\ind^{\rm op}_{\mathcal{T}}(K_{2}))\\
&=([T'_{0}]-\ind_{\mathcal{T}}(K_{2}))+[T_{1}]-[T'_{0}]-[T_{1}]+\ind^{\rm op}_{\mathcal{T}}(K_{2})\\
&=-(\ind_{\mathcal{T}}(K_{2})-\ind^{\rm op}_{\mathcal{T}}(K_{2})).
\end{align*}
Thus,
$$\Theta(X)=\ind_{\mathcal{T}}(X')-\ind^{\rm op}_{\mathcal{T}}(X')-(\ind_{\mathcal{T}}(K_{2})-\ind^{\rm op}_{\mathcal{T}}(K_{2})).$$
By repeating this process, we obtain
\begin{flalign*}\Theta(X)&=\ind_{\mathcal{T}}(X')-\ind^{\rm op}_{\mathcal{T}}(X')+(-1)^{n}(\ind_{\mathcal{T}}(T_{n-1})-\ind^{\rm op}_{\mathcal{T}}(T_{n-1}))\\&=\ind_{\mathcal{T}}(X')-\ind^{\rm op}_{\mathcal{T}}(X').\end{flalign*}
Therefore, we complete the proof.
\end{proof}
Let $T$ be an $n$-cluster tilting object in $\mathscr{C}$ and take $\mathcal{T}=\add T$. Set $B=\End_{\mathscr{C}}(T)^{\rm opp}$ and $C=\End_{\mathcal{A}}(T)^{\rm opp}$. In what follows, we will consider the following functors
\begin{flalign*}&\mathbb{F}: \mathscr{C}\rightarrow\mod B;~M\mapsto \Hom_{\mathscr{C}}(T,M),\\
&\mathbb{G}: \mathcal{A}\rightarrow\mod C;~M\mapsto \Hom_{\mathcal{A}}(T,M),\end{flalign*}
and $\mathbb{H}=\mathbb{G}\pi$.

\begin{lemma}\label{map} There is a map defined by
$$\Phi:K_{0}^{\rm sp}(\mod C)\rightarrow K_{0}^{\rm sp}(\mathcal{T});~[\mathbb{H}(M)]\mapsto \Theta(M),$$
for any $M\in \mathscr{C}$.
\end{lemma}
\begin{proof}By Lemma \ref{cluster}, $\mathcal{T}$ is an $n$-cluster tilting subcategory of $\mathcal{A}$. As shown in \cite{JO2}, the restriction of $\mathbb{G}$ induces an equivalence of categories
$$(\mathcal{T}\ast \Sigma\mathcal{T})/\Sigma \mathcal{T} \cong\mod C.$$
By Lemma \ref{L-31}, one can check that $\pi^{-1}(\mathcal{T}\ast \Sigma\mathcal{T})=\mathcal{T}\ast \Cone(\mathcal{T},\mathcal{P})$. Hence, for any $m\in\mod C$, there exists $M\in \mathscr{C}$ such that $\mathbb{H}(M)\cong m$. Given $M,N\in\mathscr{C}$, suppose that $\mathbb{H}(M)\cong\mathbb{H}(N)$, we need to prove $\Theta(M)=\Theta(N)$. Indeed, by the definition of $\mathbb{H}$, we have
$$M\oplus P_{M}\oplus T'_{M}\cong N\oplus P_{N}\oplus T'_{N}$$
for some $P_{M},P_{N}\in \mathcal{P}$ and $T'_{M},T'_{N}\in\Cone(\mathcal{T},\mathcal{P})$. By Proposition \ref{2333}, we have $\Theta(P_M)=\Theta(P_N)=0$ and $\Theta(T'_{M})=\Theta(T'_{N})=0$, then we finish the proof.
\end{proof}

\begin{lemma}\label{L-5} Let
$$A\stackrel{f}{\longrightarrow}B\stackrel{g}{\longrightarrow}C\stackrel{}\dashrightarrow$$
be an $\mathbb{E}$-triangle in $\mathscr{C}$.

$(1)$ If $\mathbb{F}(g)$ is an epimorphism, then
$$\ind_{\mathcal{T}}(A)-\ind_{\mathcal{T}}(B)+\ind_{\mathcal{T}}(C)=0.$$

$(2)$ If $\mathbb{H}(g)$ is an epimorphism, then $\mathbb{F}(g)$ is an epimorphism.

$(2')$ If $\mathbb{H}(f)$ is a monomorphism, then $\mathbb{F}(f)$ is a monomorphism.
\end{lemma}
\begin{proof}  (1) Take an  $\mathbb{E}$-triangle
$$K_{1}\stackrel{}{\longrightarrow}T_{0}\stackrel{\beta}{\longrightarrow}B\stackrel{}\dashrightarrow$$
such that $\beta$ is a right $\mathcal{T}$-approximation and $K_{1}\in \mathcal{T}^{\perp_{1}}$. Thus we get a commutative diagram
\begin{equation*}
\xymatrix{
  K_{1} \ar@{=}[d] \ar[r]^-{} & X \ar[d]^-{} \ar[r] & A \ar[d] \\
  K_{1}\ar[r]^-{} & T_{0} \ar_-{\alpha}[d] \ar[r]^-{\beta} &B\ar[d]^-{g} \\
   & C\ar@{=}[r] & C.  }
\end{equation*}
By hypothesis, $\mathbb{F}(g)$ is an epimorphism and thus $\alpha$ is a right $\mathcal{T}$-approximation. By Proposition \ref{useful}, we obtain
$$\ind_{\mathcal{T}}(X)=\ind_{\mathcal{T}}(K_{1})+\ind_{\mathcal{T}}(A),[T_{0}]=\ind_{\mathcal{T}}(X)+\ind_{\mathcal{T}}(C),[T_{0}]=\ind_{\mathcal{T}}(K_{1})+\ind_{\mathcal{T}}(B).$$
Hence, we have
\begin{align*}
&\ind_{\mathcal{T}}(A)-\ind_{\mathcal{T}}(B)+\ind_{\mathcal{T}}(C)\\&=(\ind_{\mathcal{T}}(X)-\ind_{\mathcal{T}}(K_{1}))-([T_{0}]-\ind_{\mathcal{T}}(K_{1}))+\ind_{\mathcal{T}}(C)\\
&=\ind_{\mathcal{T}}(X)-[T_{0}]+\ind_{\mathcal{T}}(C)
=0.
\end{align*}

$(2)$ Since $\mathbb{H}(g)$ is an epimorphism, for any $a\in\Hom_{\mathscr{C}}(T,C)$, there exists $h\in\Hom_{\mathscr{C}}(T,B)$ such that $a-gh$ factors through some projective object $P:$
$$\xymatrix{
   T\ar[rr]^{a-gh} \ar[dr]_{t_{1}}
                &  &      C  \\
                &     P   \ar[ur]_{t_{2}}           }
$$
Since $P$ is projective, there exists $s\in\Hom_{\mathscr{C}}(P,B)$ such that $t_{2}=gs$. It follows that $a=g(st_{1}+h)$ and thus $\mathbb{F}(g)$ is an epimorphism. The proof for $(2')$ is analogous.
\end{proof}

The following proposition is inspired by  \cite[Theorem 4.4]{JO2}.
\begin{proposition}\label{L-7} For any $\mathbb{E}$-triangle
$$A\stackrel{f}{\longrightarrow}B\stackrel{g}{\longrightarrow}C\stackrel{}\dashrightarrow$$
in $\mathscr{C}$, we have
$$ \ind_{\mathcal{T}}(A)-\ind_{\mathcal{T}}(B)+\ind_{\mathcal{T}}(C)=\Phi({\rm coker} \mathbb{H}(g)).$$
\end{proposition}
\begin{proof}  By Lemma \ref{L-31}, \begin{equation*}
A\stackrel{\pi(f)}{\longrightarrow}B\stackrel{\pi(g)}{\longrightarrow}C\stackrel{s}{\longrightarrow}\Sigma A
\end{equation*}
is a triangle in $\mathcal{A}$. By Lemma \ref{cluster} and \cite[Theorem 3.1]{IY}, $C\in\mathcal{T}\ast \Sigma\mathcal{T}\ast\cdots \ast\Sigma^{n-1}\mathcal{T}$.

If $\mathcal{T}$ is a $2$-cluster tilting subcategory, i.e. $n=2$, by \cite[Lemma 2.4]{JO2}, there is a decomposition
$$\xymatrix{
   C\ar[rr]^{s} \ar[dr]_{t_{1}}
                &  &      \Sigma A \\
                &     M  \ar[ur]_{t_{2}}           }
$$
such that $M\in \mathcal{T}\ast \Sigma\mathcal{T}$, $\mathbb{G}(t_{1})$ is an epimorphism and  $\mathbb{G}(t_{2})$ is a monomorphism. Thus, $${\rm coker} \mathbb{H}(g)={\rm coker}\mathbb{G}\pi(g)\cong\Im \mathbb{G}(s)\cong\mathbb{G}(M)=\mathbb{H}(M).$$
Applying the octahedral axiom, we have a commutative diagram of triangles in $\mathcal{A}$
$$
\xymatrix{
   &  E\ar[d]_{} \ar@{=}[r]^{} &  E\ar[d]_{} & \\
  A\ar@{=}[d]_{} \ar[r]^{} &  B\ar[d]_{} \ar[r]^{} &  C\ar[d]_-{t_{1}} \ar[r]^-{s} &\Sigma A  \ar@{=}[d]^{} \\
   A\ar[r]^{} &  V\ar[d]_-{t_{5}} \ar[r]^-{t_{3}} &  M\ar[d]_-{t_{4}} \ar[r]^-{t_{2}} & \Sigma A\\
   &  \Sigma E \ar@{=}[r]^{} &  \Sigma E &    }
$$
By the lift of triangles, we have the following $\mathbb{E}$-triangles
\begin{flalign*}&\Sigma^{-1}M\stackrel{}{\longrightarrow}I(\Sigma^{-1}M)\stackrel{}{\longrightarrow}M\stackrel{}\dashrightarrow,\quad
E\stackrel{}{\longrightarrow}I(E)\stackrel{}{\longrightarrow}\Sigma E\stackrel{}\dashrightarrow,\\
&\Sigma^{-1}M\stackrel{}{\longrightarrow}A\oplus I(\Sigma^{-1}M)\stackrel{\alpha}{\longrightarrow}V\stackrel{}\dashrightarrow,\quad
E\stackrel{}{\longrightarrow}B\oplus I(E)\stackrel{\beta}{\longrightarrow}V\stackrel{}\dashrightarrow,\end{flalign*}
and $$E\stackrel{}{\longrightarrow}C\oplus I(E)\stackrel{\gamma}{\longrightarrow}M\stackrel{}\dashrightarrow,$$ where $I(\Sigma^{-1}M)$, $I(E)\in \mathcal{P}$. Since $\mathbb{G}(t_{2})$ is a monomorphism, we get $\mathbb{H}(\alpha)$ is an epimorphism. Thus, by Lemma \ref{L-5}, we have
$$\ind_{\mathcal{T}}(A\oplus I(\Sigma^{-1}M))=\ind_{\mathcal{T}}(\Sigma^{-1}M)+\ind_{\mathcal{T}}(V).$$
Similarly, we can prove that $\mathbb{H}(\beta)$ and $\mathbb{H}(\gamma)$ are epimorphisms. Thus,
\begin{flalign*}&\ind_{\mathcal{T}}(B\oplus I(E))=\ind_{\mathcal{T}}(E)+\ind_{\mathcal{T}}(V),\\
&\ind_{\mathcal{T}}(C\oplus I(E))=\ind_{\mathcal{T}}(E)+\ind_{\mathcal{T}}(M).\end{flalign*}
Hence, we obtain
\begin{align*}
&\ind_{\mathcal{T}}(A)-\ind_{\mathcal{T}}(B)+\ind_{\mathcal{T}}(C)\\&=\ind_{\mathcal{T}}(\Sigma^{-1}M)+\ind_{\mathcal{T}}(V)-[ I(\Sigma^{-1}M)]-\ind_{\mathcal{T}}(E)\\
&\quad-\ind_{\mathcal{T}}(V)+[I(E)]+\ind_{\mathcal{T}}(E)+\ind_{\mathcal{T}}(M)-[I(E)]\\
&=\ind_{\mathcal{T}}(\Sigma^{-1}M)-[I(\Sigma^{-1}M)]+\ind_{\mathcal{T}}(M)\\
&=\Theta(M)=\Phi(\mathbb{H}(M))=\Phi({\rm coker} \mathbb{H}(g)).
\end{align*}
For the case $n>2$, by \cite[Lemma 1.7]{JO2}, there is a triangle
 \begin{equation*}
X\stackrel{c}{\longrightarrow}C\stackrel{}{\longrightarrow}M\stackrel{}{\longrightarrow}\Sigma X
\end{equation*}
with $X\in\mathcal{T}\ast \Sigma\mathcal{T}$ and $M\in\Sigma^{2}\mathcal{T}\ast\cdots \ast\Sigma^{n-1}\mathcal{T}$. Then we get a commutative diagram  of triangles
$$
\xymatrix{
   &  \Sigma^{-1}M\ar[d]_{} \ar@{=}[r]^{} &  \Sigma^{-1}M\ar[d]_{} & \\
  A\ar@{=}[d]_{} \ar[r]^{} &  L\ar[d]_{b} \ar[r]^{a} &  X\ar[d]_-{c} \ar[r]^-{} &\Sigma A  \ar@{=}[d]^{} \\
   A\ar[r]^{} &  B\ar[d]_-{} \ar[r]^-{\pi(g)} &  C\ar[d]_-{} \ar[r]^-{} & \Sigma A\\
   & M\ar@{=}[r]^{} &  M &    }
$$
By the lift of triangles, we get the $\mathbb{E}$-triangles
\begin{flalign*}\Sigma^{-1}M\stackrel{}{\longrightarrow}I(\Sigma^{-1}M)\stackrel{}{\longrightarrow}M\stackrel{}\dashrightarrow,
\quad\Sigma^{-1}M\stackrel{}{\longrightarrow}L\oplus I(\Sigma^{-1}M)\stackrel{d}{\longrightarrow}B\stackrel{}\dashrightarrow,\end{flalign*}
and $$\Sigma^{-1}M\stackrel{}{\longrightarrow}X\oplus I(\Sigma^{-1}M)\stackrel{e}{\longrightarrow}C\stackrel{}\dashrightarrow.$$
Since $X\in\mathcal{T}\ast \Sigma\mathcal{T}$, we have
\begin{equation}\label{777}
\ind_{\mathcal{T}}(A)-\ind_{\mathcal{T}}(L)+\ind_{\mathcal{T}}(X)=\Phi({\rm coker} \mathbb{H}(a)).
\end{equation}
By \cite[Lemma 2.2]{JO2}, $\mathbb{G}(M)=0$ and thus $\mathbb{H}(d)=\mathbb{G}(b)$ is an epimorphism. Similarly, $\mathbb{H}(e)=\mathbb{G}(c)$ is an epimorphism. By Lemma \ref{L-5}, we obtain
\begin{flalign}
&\ind_{\mathcal{T}}(L\oplus I(\Sigma^{-1}M))=\ind_{\mathcal{T}}(\Sigma^{-1}M)+\ind_{\mathcal{T}}(B),\label{888}\\
&\ind_{\mathcal{T}}(X\oplus I(\Sigma^{-1}M))=\ind_{\mathcal{T}}(\Sigma^{-1}M)+\ind_{\mathcal{T}}(C)\label{999}.
\end{flalign}
Combining the equations (\ref{777}-\ref{999}), we get
$$ \ind_{\mathcal{T}}(A)-\ind_{\mathcal{T}}(B)+\ind_{\mathcal{T}}(C)=\Phi({\rm coker} \mathbb{H}(a)).$$
Again by \cite[Lemma 2.2]{JO2}, $\mathbb{G}( \Sigma^{-1}M)=0$ and thus
$${\rm coker} \mathbb{H}(a)\cong{\rm coker} \mathbb{G}(\pi(g))={\rm coker} \mathbb{H}(g).$$
Therefore, we complete the proof.
\end{proof}

\begin{remark} Proposition \ref{L-7} was proved in triangulated categories by \cite[Theorem 4.4]{JO2}. In particular, if $n=2$, it recovers \cite[Proposition 2.2]{Pa}.
\end{remark}

Let us state the main result in this section as the following
\begin{theorem}\label{main1}  The map $\Phi$ descends to the Grothendieck group $K_{0}(\mod C).$
\end{theorem}
\begin{proof} Let
$$ 0\stackrel{}{\longrightarrow}x\stackrel{a}{\longrightarrow}l\stackrel{b}{\longrightarrow}m \stackrel{}{\longrightarrow}0$$
be an exact sequence in $\mod C$. By \cite[Lemma 2.7]{JO2}, we can lift this exact sequence to a triangle in $\mathcal{A}$
\begin{equation}\label{d-1}
X\stackrel{a_{1}}{\longrightarrow}L\stackrel{b_{1}}{\longrightarrow}M{\longrightarrow}\Sigma X.
\end{equation}
By Lemma \ref{L-31}, this triangle can be lifted to an $\mathbb{E}$-triangle in $\mathscr{C}$
$$X\stackrel{{a_{2}\choose a_{3}}}{\longrightarrow}L\oplus I\stackrel{(b_{2}~b_{3})}{\longrightarrow}M\stackrel{}\dashrightarrow$$
with $I\in\mathcal{P}$. By \cite[Lemmas 4.14, 4.15]{Hu}, there exists a commutative diagram of  $\mathbb{E}$-triangles
$$
\xymatrix{
  X'\ar[d]_{} \ar[r]^{f'} & L'\ar[d]_{} \ar[r]^{g'} &M'\ar[d]_{} \ar@{-->}[r]^{} &   \\
  P_{X}\ar[d]_{} \ar[r]^{} &  P_{X}\oplus P_{M}  \ar[d]_{} \ar[r]^{} &  P_{M}\ar[d]_{} \ar@{-->}[r]^{} &   \\
   X\ar@{-->}[d]_{} \ar[r]^-{{a_{2}\choose a_{3}}} &  L\oplus I\ar@{-->}[d]_{} \ar[r]^-{(b_{2}~b_{3})} &  M\ar@{-->}[d]_{} \ar@{-->}[r]^{} &  \\
  & &  &     }
$$
with $P_{X},P_{M}\in \mathcal{P}$.
Since  $\mathbb{H}((b_{2}~b_{3}))=b$ is an epimorphism, by Lemma \ref{L-5}, we obtain
\begin{equation}\label{e1}
\ind_{\mathcal{T}}(X)-\ind_{\mathcal{T}}(L\oplus I)+\ind_{\mathcal{T}}(M)=0.
\end{equation}
 Note that  $\pi(f')=\Sigma^{-1}a_{1}$ and $\pi(g')=\Sigma^{-1}b_{1}$ in $\mathcal{A}$. Thus, the $\mathbb{E}$-triangle
$$X'\stackrel{f'}{\longrightarrow}L'\stackrel{g'}{\longrightarrow}M'\stackrel{}\dashrightarrow$$
is the lift of the triangle
$$ \Sigma^{-1}X\stackrel{\Sigma^{-1}a_{1}}{\longrightarrow}\Sigma^{-1}L\stackrel{\Sigma^{-1}b_{1}}{\longrightarrow}\Sigma^{-1}M \stackrel{}{\longrightarrow}X.$$
Applying the functor $\Hom_{\mathcal{A}}(T,-)$ to the triangle (\ref{d-1}), we get an exact sequence
\begin{equation}\label{exact}
\Hom_{\mathcal{A}}(T,\Sigma^{-1}L)\stackrel{\mathbb{G}(\Sigma^{-1}b_{1})}{\longrightarrow}\Hom_{\mathcal{A}}(T,\Sigma^{-1}M)\stackrel{}{\longrightarrow}\Hom_{\mathcal{A}}(T,X) \stackrel{\mathbb{G}(a_{1})}{\longrightarrow}\Hom_{\mathcal{A}}(T,L).
\end{equation}
Since $\mathbb{G}(a_{1})=a$ is a monomorphism, by the exactness of $(\ref{exact})$, we get
 $${\rm coker }\mathbb{H}(g')={\rm coker }(\mathbb{G}(\Sigma^{-1}b_{1}))\cong\ker(\mathbb{G}(a_{1}))=0$$
and thus $\mathbb{H}(g')$ is an epimorphism. By Lemma \ref{L-5},
\begin{equation}\label{e2}
\ind_{\mathcal{T}}(X')-\ind_{\mathcal{T}}(L')+\ind_{\mathcal{T}}(M')=0.
\end{equation}
Set $\eta_{XLM}:=\Theta(X)-\Theta(L)+\Theta(M)$, by Proposition \ref{2333},
$$\eta_{XLM}=\Theta(X)-\Theta(L\oplus I)+\Theta(M).$$
Hence, by (\ref{e1}) and (\ref{e2}), we obtain
\begin{align*}
\eta_{XLM}&=\ind_{\mathcal{T}}(X)-[P_{X}]+\ind_{\mathcal{T}}(X')-\ind_{\mathcal{T}}(L\oplus I)+[P_{X}\oplus P_{M}]\\
&\quad-\ind_{\mathcal{T}}(L')+\ind_{\mathcal{T}}(M)-[P_{M}]+\ind_{\mathcal{T}}(M')\\
&=\ind_{\mathcal{T}}(X)-\ind_{\mathcal{T}}(L\oplus I)+\ind_{\mathcal{T}}(M)+\ind_{\mathcal{T}}(X')-\ind_{\mathcal{T}}(L')+\ind_{\mathcal{T}}(M')\\
&=0.
\end{align*}
Therefore, we finish the proof.
\end{proof}
\begin{corollary}
If $\mathscr{C}$ is a triangulated category with the suspension functor ${\Sigma}$ and $\mathcal{T}=\add T$ is an $n$-cluster tilting subcategory of $\mathscr{C}$, then we have the map
$$\Phi: K_0(\mod B)\longrightarrow  K_{0}^{\rm sp}(\mathcal{T}), \Hom_{\mathscr{C}}(T,M)\mapsto\ind_{\mathcal{T}}(\Sigma^{-1}M)+\ind_{\mathcal{T}}(M).$$
\end{corollary}
Note that for $n=2$ the above map was used to provide the exponents of the cluster characters in 2-CY triangulated categories (cf. \cite{Pa}).

\section{Cluster characters for 2-CY categories}
The notion of cluster characters for 2-CY triangulated categories is due to Palu \cite{Pa}. Fu and Keller \cite{Chan} established the link of $2$-CY Frobenius exact categories and cluster algebras with coefficients via cluster characters. In this section, by applying the obtained results in Section 3, we define the cluster characters on 2-CY Frobenius extriangulated categories and establish the multiplication formulas of cluster characters.


\begin{definition}\label{CC3} Let $\mathscr{C}$ be a {2-CY Frobenius extriangulated category}. A {\em cluster character} on $\mathscr{C}$ with values in a commutative ring $R$ is a map $\varphi:{\rm obj}(\mathscr{C})\rightarrow R$ satisfying the following conditions:

$(1)$ If $X\cong Y$ in $\mathscr{C}$, then $\varphi(X)=\varphi(Y)$.

$(2)$ For any objects $X,Y\in\mathscr{C}$, we have $\varphi(X\oplus Y)=\varphi(X)\varphi(Y)$.

$(3)$ For any objects $M,N\in\mathscr{C}$, if $\dim_{k}\mathbb{E}(N,M)=1$ and
$$N\stackrel{}{\longrightarrow}L\stackrel{}{\longrightarrow}M\stackrel{}\dashrightarrow~\text{and}~M\stackrel{}{\longrightarrow}L'\stackrel{}{\longrightarrow}N\stackrel{}\dashrightarrow$$
are  nonsplit $\mathbb{E}$-triangles, then $\varphi(N)\varphi(M)=\varphi(L)+\varphi(L')$.
\end{definition}
In what follows, the following assumptions are fixed.

$\bullet$  $\mathscr{C}$ is a {2-CY Frobenius extriangulated category}. Then, by Lemma \ref{L-31}, the stable category $\mathcal{A}$ of $\mathscr{C}$ is a 2-CY triangulated category with the suspension functor $\Sigma$.

$\bullet$ $T=\oplus_{i=1}^{n}T_{i}$ is a $2$-cluster tilting object in $\mathscr{C}$. We assume that $T_{i}$ is projective exactly for $r+1\leq i\leq n$. By Lemma \ref{cluster}, $\pi(T)=\oplus_{i=1}^{r}T_{i}=:T'$ is a $2$-cluster tilting object in $\mathcal{A}$.

Recall that $B=\End_{\mathscr{C}}(T)^{\rm opp}$, $C=\End_{\mathcal{A}}(T)^{\rm opp}$, and we have the following functors
\begin{flalign*}&\mathbb{F}: \mathscr{C}\rightarrow\mod B;~M\mapsto \Hom_{\mathscr{C}}(T,M),\\
&\mathbb{G}: \mathcal{A}\rightarrow\mod C;~M\mapsto \Hom_{\mathcal{A}}(T,M),\end{flalign*}
and $\mathbb{H}=G\pi$.

$\bullet$ For any $M\in\mathscr{C}$, $\mathbb{G}(\Sigma M)$ is finite dimensional.

For each $M\in\mathscr{C}$, we write
$$\ind_{T}(M):=\ind_{{\rm add}\, T}(M)~\text{and}~\ind_{T'}(M):=\ind_{{\rm add}\,T'}(\pi(M)).$$
We denote by $[\ind_{T}(M):T_{i}]$ the $i$th coefficient of $\ind_{T}(M)$ with respect to the basis $[T_1],\cdots,[T_n]$. Similarly, $[\ind_{T'}(M):T_{i}]$ is the $i$th coefficient of $\ind_{T'}(M)$ with respect to the basis $[T_1],\cdots,[T_r]$.

We recall the cluster character formulas in 2-CY triangulated categories and 2-CY Frobenius exact categories from \cite{Pa} and \cite{Chan}, respectively.

\subsection{Cluster characters for 2-CY triangulated categories} In this subsection, we assume that $\mathscr{C}$ is Hom-finite. Recall that $\mathcal{A}$ is a 2-CY triangulated category with the $2$-cluster tilting object $T'$. We define a bilinear form
$$\lr{~,~}:K_{0}^{\rm sp}(\mod C)\times K_{0}^{\rm sp}(\mod C)\rightarrow \mathbb{Z }$$
by setting
$$\lr{M,N}={\rm dim}_k\Hom_C(M,N)-{\rm dim}_k\Ext_C^1(M,N)$$
for any $C$-modules $M,N$. Similarly, we can also define the bilinear form $\lr{~,~}$ on $K_{0}^{\rm sp}(\mod B)$. We define an antisymmetric bilinear form on $K_{0}^{\rm sp}(\mod C)$ by setting
$$\lr{M,N}_{a}=\lr{M,N}-\lr{N,M}$$ for any $C$-modules $M,N$.
By \cite[Theorem 3.4]{Pa},  the antisymmetric bilinear form $\lr{~,~}_a$ descends to
the Grothendieck group $K_{0}(\mod C)$. For any $M\in \mathcal{A}$, set $${\rm coind}_{T'}(M)=-{\rm ind}_{T'} (\Sigma^{-1}M).$$  For each $1\leq i\leq r$, let $S_{i}$ be the top of the projective $C$-module $\mathbb{G}(T_{i})$.
Let ${\rm indec}\,(\mathcal{A})$ be a set of representatives for the isomorphism classes of indecomposable
objects in $\mathcal{A}$.

According to \cite[Section 6]{Ca}, the {\em Caldero--Chapoton map} $X^{T'}_{?}:{\rm indec}\,(\mathcal{A})\rightarrow \mathbb{Q}(x_{1},x_{2},\cdots,x_{r})$ is defined by
$$X^{T'}_{M}=\left\{
\begin{aligned}
~x_{i}, &  & \text{if}~M\cong \Sigma T_{i}, \\
\sum_{e}\chi({\rm Gr}_{e}(\mathbb{G}(M)))\prod_{i=1}^{r}x_{i}^{\lr{S_{i},e}_{a}-\lr{S_{i},\mathbb{G}M}}, &  & \text{otherwise}
\end{aligned}
\right.
$$
and it can be extended to a map $X^{T'}_{?}:{\rm obj}(\mathcal{A})\rightarrow \mathbb{Q}(x_{1},x_{2},\cdots,x_{r})$ by requiring that $X_{M\oplus N}^{T'}=X_{M}^{T'}X_{ N}^{T'}$.
Here ${\rm Gr}_{e}(\mathbb{G}(M))$ denotes the variety of submodules $N$ of $\mathbb{G}(M)$ whose class in $K_{0}(\mod C)$ is $e$ and $\chi$ is the Euler--Poincar\'{e} characteristic. By \cite[Theorem 1.4]{Pa}, $X^{T'}_{?}$ is a cluster character.
In fact, by \cite[Lemma 4.1]{Pa}, for any $M\in\A$,
\begin{equation}\label{sanjiaotz}
X_M^{T'}=\prod_{i=1}^{r}x_{i}^{-[{\rm coind}_{T'}M:T_i]}\sum_{e}\chi({\rm Gr}_{e}(\mathbb{G}(M)))\prod_{i=1}^{r}x_{i}^{\lr{S_{i},e}_{a}}.\end{equation}

\subsection{Cluster characters for 2-CY Frobenius exact categories} In this subsection, we assume that $\mathscr{C}$ is Hom-finite and exact. For each $1\leq i\leq n$, let $S_{i}$ be the top of the projective $B$-module $\mathbb{F}(T_{i})$. We  identify $\mod C$ with the full subcategory of $\mod B$ formed by the modules without composition factors isomorphic to one of the $S_{i}$, $r+1\leq i\leq n$. We define a bilinear form
$$\lr{~,~}_{3}:K_{0}^{\rm sp}(\mod B)\times K_{0}^{\rm sp}(\mod B)\rightarrow \mathbb{Z }$$
by setting
$$\lr{M,N}_{3}=\sum_{i=0}^{3}(-1)^{i}{\rm dim}_k\Ext_B^i(M,N)$$
for any $B$-modules $M$ and $N$. By \cite[Proposition 3.2]{Chan}, if $M$ is a $C$-module, then $\lr{M,N}_{3}$ only depends on the class ${\bf dim}\,M$ of $M$ in $K_{0}(\mod C)$. Hence, for any $C$-module $M$ with ${\bf dim}\,M=e$, we set
$$\lr{e,N}_{3}:=\lr{M,N}_{3}$$ for any $B$-module $N$.

Given $M\in \mathscr{C}$, Fu and Keller \cite{Chan} defined the Laurent polynomial
\begin{equation}\label{fucharacter}
X'^{T}_{M}=\prod_{i=1}^{n}x_{i}^{\lr{\mathbb{F}M,S_{i}}}\sum_{e}\chi({\rm Gr}_{e}(\Ext_{\mathscr{C}}^{1}(T,M)))\prod_{i=1}^{n}x_{i}^{-\lr{e,S_{i}}_{3}}
\end{equation}
where $\Ext_{\mathscr{C}}^{1}(T,M)\cong\Hom_{\mathcal{A}}(T,\Sigma M)$ is a $C$-module, ${\rm Gr}_{e}(\Ext_{\mathscr{C}}^{1}(T,M))$ denotes the variety of submodules $N$ of $\Ext_{\mathscr{C}}^{1}(T,M)$ such that the class of $N$ in $K_{0}(\mod C)$ is $e$ and $\chi$ is the Euler--Poincar\'{e} characteristic. By \cite[Theorem 3.3]{Chan}, $X'^{T}_{?}$ is a cluster character on $\mathscr{C}$. Moreover, $X'^{T}_{M}(x_1,\cdots,x_r,1,\cdots,1)=X^{T'}_{\Sigma M}$ for any $M \in\mathscr{C}$.

\subsection{Cluster characters for 2-CY Frobenius extriangulated categories}
By Theorem \ref{main1}, for any $M\in \mathscr{C}$, $\Theta(M)$ only depends on the class ${\bf dim}\,\mathbb{H}(M)=:e$ in $K_{0}(\mod C).$ Hence, we put
$[e:T_{i}]:=[\Theta(M):T_{i}].$
\begin{definition}\label{CC} Define a map $\mathbb{X}^{T}_{?}:{\rm obj}(\mathscr{C})\rightarrow \mathbb{Q}(x_{1},x_{2},\cdots,x_{n})$ by
\begin{equation}\label{echaracter}
\mathbb{X}^{T}_{M}=\prod_{i=1}^{n}x_{i}^{[{\rm ind}_{T}(M):T_{i}]}\sum_{e}\chi({\rm Gr}_{e}(\mathbb{G}(\Sigma M)))\prod_{i=1}^{n}x_{i}^{-[e:T_{i}]}\end{equation}
for any $M\in\mathscr{C}$.
Here $\mathbb{G}(\Sigma M)$ is a $C$-module, and ${\rm Gr}_{e}(\mathbb{G}(\Sigma M))$ denotes the variety of submodules $N$ of $\mathbb{G}(\Sigma M)$ such that the class of $N$ in $K_{0}(\mod C)$ is $e$ and $\chi$ is the Euler--Poincar\'{e} characteristic.
\end{definition}
\begin{lemma}\label{L-32}
$(1)$ For each $1\leq i\leq n$, $\mathbb{X}^{T}_{T_{i}}=x_{i}$.
$(2)$ For any $M,N\in\mathscr{C}$, $\mathbb{X}^{T}_{M\oplus N}=\mathbb{X}^{T}_{M}\mathbb{X}^{T}_{N}.$
\end{lemma}
\begin{proof} (1) This is straightforward.

$(2)$ As shown in \cite[Lemma 4.1]{Pa}, we have
$$\chi({\rm Gr}_{e}(\mathbb{G}(\Sigma M\oplus\Sigma  N)))=\sum_{e_{1}+e_{2}=e}\chi({\rm Gr}_{e_{1}}(\mathbb{G}(\Sigma M)))\chi({\rm Gr}_{e_{2}}(\mathbb{G}(\Sigma N))).$$
Thus,
\begin{align*}
\mathbb{X}^{T}_{M\oplus N}&=\prod_{i=1}^{n}x_{i}^{[{\rm ind}_{T}(M\oplus N):T_{i}]}\sum_{e}\chi({\rm Gr}_{e}(\mathbb{G}(\Sigma (M\oplus N)))\prod_{i=1}^{n}x_{i}^{-[e:T_{i}]}\\
&=\prod_{i=1}^{n}x_{i}^{[{\rm ind}_{T}(M\oplus N):T_{i}]}\sum_{e}\sum_{e_{1}+e_{2}=e}\chi({\rm Gr}_{e_{1}}(\mathbb{G}(\Sigma M)))\chi({\rm Gr}_{e_{2}}(\mathbb{G}(\Sigma N)))\prod_{i=1}^{n}x_{i}^{-[e:T_{i}]}\\
&=\prod_{i=1}^{n}x_{i}^{[{\rm ind}_{T}(M):T_{i}]}\prod_{i=1}^{n}x_{i}^{[{\rm ind}_{T}(N):T_{i}]}\sum_{e}(\sum_{e_{1}+e_{2}=e}\chi({\rm Gr}_{e_{1}}(\mathbb{G}(\Sigma M)))\\
&\quad\cdot\chi({\rm Gr}_{e_{2}}(\mathbb{G}(\Sigma N)))\prod_{i=1}^{n}x_{i}^{-[e_{1}:T_{i}]}\prod_{i=1}^{n}x_{i}^{-[e_{2}:T_{i}]})\\
&=(\prod_{i=1}^{n}x_{i}^{[{\rm ind}_{T}(M):T_{i}]}\sum_{e_{1}}\chi({\rm Gr}_{e_{1}}(\mathbb{G}(\Sigma (M)))\prod_{i=1}^{n}x_{i}^{-[e_{1}:T_{i}]})\\
&\quad\cdot(\prod_{i=1}^{n}x_{i}^{[{\rm ind}_{T}(N):T_{i}]}\sum_{e_{2}}\chi({\rm Gr}_{e_{2}}(\mathbb{G}(\Sigma (N)))\prod_{i=1}^{n}x_{i}^{-[e_{2}:T_{i}]})\\
&=\mathbb{X}^{T}_{M}\mathbb{X}^{T}_{N}.
\end{align*}
\end{proof}
\begin{theorem} The map $\mathbb{X}^{T}_{?}:{\rm obj}(\mathscr{C})\rightarrow \mathbb{Q}(x_{1},x_{2},\cdots,x_{n})$ is a cluster character.
\end{theorem}
\begin{proof}

Let $M,N\in\mathscr{C}$ and assume that $\dim_{k}\mathbb{E}(N,M)=1$. Then we have non-split $\mathbb{E}$-triangles
$$N\stackrel{a}{\longrightarrow}L\stackrel{b}{\longrightarrow}M\stackrel{}\dashrightarrow$$
and
$$M\stackrel{c}{\longrightarrow}L'\stackrel{d}{\longrightarrow}N\stackrel{}\dashrightarrow.$$
By Lemma \ref{L-32}, we only need to show that $\mathbb{X}^{T}_{N}\mathbb{X}^{T}_{M}=\mathbb{X}^{T}_{L}+\mathbb{X}^{T}_{L'}$. By Lemma \ref{L-31}, we have triangles
\begin{equation}\label{d-2}
\Sigma N\stackrel{\Sigma \pi(a)}{\longrightarrow}\Sigma L\stackrel{\Sigma \pi(b)}{\longrightarrow}\Sigma M\stackrel{}{\longrightarrow}\Sigma^{2} N
\end{equation}
and
$$\Sigma M\stackrel{\Sigma \pi(c)}{\longrightarrow}\Sigma L'\stackrel{\Sigma \pi(d)}{\longrightarrow}\Sigma N\stackrel{}{\longrightarrow}\Sigma^{2} M$$
in $\mathcal{A}$. For any $e,f,g\in K_{0}(\mod C)$, set
$$X_{e,f}=\{E\subseteq \mathbb{G}(\Sigma L)~|~{\bf dim}\, (\mathbb{G}\Sigma \pi(a))^{-1}E=e~\text{ and}~{\bf dim}\, (\mathbb{G}\Sigma \pi(b))E=f\},$$
$$Y_{e,f}=\{E\subseteq \mathbb{G}(\Sigma L')~|~{\bf dim}\,(\mathbb{G}\Sigma \pi(c))^{-1}E=f~\text{ and}~{\bf dim}\, (\mathbb{G}\Sigma \pi(d))E=e\},$$
$$X^{g}_{e,f}=X_{e,f}\cap {\rm Gr}_{g}\mathbb{G}(\Sigma L)~\text{and}~Y^{g}_{e,f}=Y_{e,f}\cap {\rm Gr}_{g}\mathbb{G}(\Sigma L').$$
By \cite[Section 5]{Pa}, we have
\begin{equation}\label{111}
\begin{split}
&\chi({\rm Gr}_{g}\mathbb{G}(\Sigma L))=\sum\limits_{e,f}\chi(X^{g}_{e,f}),\quad\chi({\rm Gr}_{g}\mathbb{G}(\Sigma L'))=\sum\limits_{e,f}\chi(Y^{g}_{e,f}),\\
&\chi({\rm Gr}_{e}(\mathbb{G}(\Sigma N))\times {\rm Gr}_{f}(\mathbb{G}(\Sigma  M)))=\sum_{g}(\chi(X^{g}_{e,f})+\chi(Y^{g}_{e,f})).
\end{split}\end{equation}
$\mathbf{Claim }$ If $X^{g}_{e,f}\neq {\O}$, then
\begin{equation}\label{222}
[\ind_{T}(N)-\ind_{T}(L)+\ind_{T}(M):T_{i}]=[f+e-g:T_{i}].
\end{equation}
Indeed, by \cite[Section 4]{Pa}, for $E\in X^{g}_{e,f}$, there is a  commutative diagram
\begin{equation*}
\xymatrix{
 &  (\mathbb{G}\Sigma \pi(a))^{-1}E\ar[d]^{i} \ar[r]^-{\alpha} & E\ar[d]^{j} \ar[r]^-{\beta} & (\mathbb{G}\Sigma\pi(b))E\ar[d]^{k}  &  \\
&  \mathbb{G}(\Sigma N) \ar[r]^-{\mathbb{G}\Sigma \pi(a)} &  \mathbb{G}(\Sigma L )\ar[r]^-{\mathbb{G}\Sigma \pi(b)} &\mathbb{G}(\Sigma M)  &  }
\end{equation*}
where $i,j,k$ are monomorphisms, $\beta$ is an epimorphism. It follows that $\ker\mathbb{G}\Sigma \pi(a)\cong\ker \alpha$, ${\bf dim}\,(\mathbb{G}\Sigma \pi(a))^{-1}E=e$ and ${\bf dim}\,(\mathbb{G}\Sigma\pi(b))E=f.$  By the exact sequence
$$0\stackrel{}{\longrightarrow}\ker \alpha\stackrel{}{\longrightarrow} (\mathbb{G}\Sigma \pi(a))^{-1}E \stackrel{\alpha}{\longrightarrow}E\stackrel{\beta}{\longrightarrow}(\mathbb{G}\Sigma\pi(b))E\stackrel{}{\longrightarrow}0,$$
we obtain $g=f+e-{\bf dim}\ker\alpha$ and $$[g:T_{i}]=[f+e:T_{i}]-[{\bf dim}\ker\alpha:T_{i}]=[f+e:T_{i}]-[{\bf dim}\ker \mathbb{G}\Sigma \pi(a):T_{i}].$$
Applying the functor $\Hom_{\mathcal{A}}(T,-)$ to the triangle (\ref{d-2}), we get an exact sequence
\begin{equation*}
\Hom_{\mathcal{A}}(T,L)\stackrel{s}{\longrightarrow}\Hom_{\mathcal{A}}(T,M)\stackrel{}{\longrightarrow}\Hom_{\mathcal{A}}(T,\Sigma N) \stackrel{\mathbb{G}\Sigma \pi(a)}{\longrightarrow}\Hom_{\mathcal{A}}(T,\Sigma L).
\end{equation*}
Thus, ${\rm coker}(s)\cong\ker \mathbb{G}\Sigma \pi(a)\cong\ker\alpha $.  Since $\mathbb{H}(b)=s$, by Proposition \ref{L-7}, we get
$$\ind_{T}(N)-\ind_{T}(L)+\ind_{T}(M)=\Phi({{\rm coker} \mathbb{H}(b)})=\Phi(\ker \mathbb{G}\Sigma \pi(a)).$$
Thus,
$$[\ind_{T}(N)-\ind_{T}(L)+\ind_{T}(M):T_{i}]=[{\bf dim}\ker\mathbb{G}\Sigma \pi(a):T_{i}]=[f+e-g:T_{i}].$$
Similarly, if $Y^{g}_{e,f}\neq {\O}$, then
\begin{equation}\label{666}
[\ind_{T}(N)-\ind_{T}(L')+\ind_{T}(M):T_{i}]=[f+e-g:T_{i}].
\end{equation}
By (\ref{111}-\ref{666}), we obtain
\begin{align*}
\mathbb{X}^{T}_{N}\mathbb{X}^{T}_{M}&=(\prod_{i=1}^{n}x_{i}^{[{\rm ind}_{T}(N):T_{i}]}\sum_{e}\chi({\rm Gr}_{e}(\mathbb{G}(\Sigma (N)))\prod_{i=1}^{n}x_{i}^{-[e:T_{i}]})\\&\quad\cdot(\prod_{i=1}^{n}x_{i}^{[{\rm ind}_{T}(M):T_{i}]}\sum_{f}\chi({\rm Gr}_{f}(\mathbb{G}(\Sigma (M)))\prod_{i=1}^{n}x_{i}^{-[f:T_{i}]})\\
&=\prod_{i=1}^{n}x_{i}^{[{\rm ind}_{T}(N)+{\rm ind}_{T}(M):T_{i}]}\sum_{e,f}\chi({\rm Gr}_{e}(\mathbb{G}(\Sigma (N))))\chi({\rm Gr}_{f}(\mathbb{G}(\Sigma (M)))\prod_{i=1}^{n}x_{i}^{-[e+f:T_{i}]}\\
&=\prod_{i=1}^{n}x_{i}^{[{\rm ind}_{T}(N)+{\rm ind}_{T}(M):T_{i}]}\sum_{e,f,g}(\chi(X^{g}_{e,f})+\chi(Y^{g}_{e,f}))\prod_{i=1}^{n}x_{i}^{-[e+f:T_{i}]} \\
&=\sum_{e,f,g}\chi(X^{g}_{e,f})\prod_{i=1}^{n}x_{i}^{[-e-f+{\rm ind}_{T}(N)+{\rm ind}_{T}(M):T_{i}]}+\sum_{e,f,g}\chi(Y^{g}_{e,f})\prod_{i=1}^{n}x_{i}^{[-e-f+{\rm ind}_{T}(N)+{\rm ind}_{T}(M):T_{i}]}\\
&=\sum_{e,f,g}\chi(X^{g}_{e,f})\prod_{i=1}^{n}x_{i}^{[-g+{\rm ind}_{T}(L):T_{i}]}+\sum_{e,f,g}\chi(Y^{g}_{e,f})\prod_{i=1}^{n}x_{i}^{[-g+{\rm ind}_{T}(L'):T_{i}]}\\
&=\prod_{i=1}^{n}x_{i}^{[{\rm ind}_{T}(L):T_{i}]}\sum_{g}\chi({\rm Gr}_{g}(\mathbb{G}(\Sigma L)))\prod_{i=1}^{n}x_{i}^{-[g:T_{i}]}\\&\quad+\prod_{i=1}^{n}x_{i}^{[{\rm ind}_{T}(L'):T_{i}]}\sum_{g}\chi({\rm Gr}_{g}(\mathbb{G}(\Sigma L')))\prod_{i=1}^{n}x_{i}^{-[g:T_{i}]}\\
&=\mathbb{X}^{T}_{L}+\mathbb{X}^{T}_{L'}.
\end{align*}
Therefore, we complete the proof.
\end{proof}

\subsection{Comparisons of cluster characters}
Assume that $\mathscr{C}$ is Hom-finite.
Since $\mathcal{A}$ is a 2-CY triangulated category, one can define the cluster character $X^{T'}_{?}$ on $\mathcal{A}$. The following proposition shows that $\mathbb{X}^{T}_{?}$ generalizes the cluster characters for 2-CY triangulated categories in \cite{Pa}.

\begin{proposition}\label{www} For any indecomposable object $M\in\mathscr{C}$, we have $$\mathbb{X}^{T}_{M}(x_1,\cdots,x_r,1,\cdots,1)=X^{T'}_{\Sigma M}.$$ In particular, if $\mathscr{C}$ is a triangulated category, then $\mathbb{X}^{T}_{M}=X^{T}_{\Sigma M}$.
\end{proposition}
\begin{proof}
If $M\cong T_i$ for some $1\leq i\leq r$, then $\mathbb{X}^{T}_{M}(x_1,\cdots,x_r,1,\cdots,1)=x_i=X^{T'}_{\Sigma M}$.
If $M\cong T_i$ for some $r+1\leq i\leq n$, then $\mathbb{X}^{T}_{M}(x_1,\cdots,x_r,1,\cdots,1)=1=X^{T'}_{\Sigma M}$.

Assume that $M\ncong T_i$ for any $1\leq i\leq n$.
For each $1\leq i\leq r$, let $S_{i}$ be the top of the projective $C$-module $\mathbb{G}(T_{i})$. By \cite[Lemma 2.3]{Pa}, we have
$$\lr{S_{i},\mathbb{F}\Sigma M}=[{\rm coind}_{T'}(\Sigma M): T_{i}]=-[\ind_{T'} (M):T_{i}].$$
Thus, by (\ref{sanjiaotz}), we obtain
$$X^{T'}_{\Sigma M}=\prod_{i=1}^{r}x_{i}^{[{\rm ind}_{T'} (M): T_{i}]}\sum_{e}\chi({\rm Gr}_{e}(\mathbb{G}(\Sigma M)))\prod_{i=1}^{r}x_{i}^{\lr{S_{i},e}_{a}}.$$
While,
$$\mathbb{X}^{T}_{M}(x_1,\cdots,x_r,1,\cdots,1)=\prod_{i=1}^{r}x_{i}^{[{\rm ind}_{T}(M):T_{i}]}\sum_{e}\chi({\rm Gr}_{e}(\mathbb{G}(\Sigma M)))\prod_{i=1}^{r}x_{i}^{-[e:T_{i}]}.$$

$\mathbf{Claim~1.}$ For any $1\leq i\leq r$, we have $[{\rm ind}_{T}(M):T_{i}]=[{\rm ind}_{T'}(M):T_{i}]$.

Let
$$T^{M}_{1}\stackrel{}{\longrightarrow}T^{M}_{0}\stackrel{}{\longrightarrow}M\stackrel{}\dashrightarrow$$ be an $\mathbb{E}$-triangle
in $\mathscr{C}$ with $T^{M}_{0},T^{M}_{1}\in \add T$.

Assume that $T^{M}_{1}=\oplus_{i=1}^n a_iT_{i}$ and $T^{M}_{0}=\oplus_{i=1}^nb_iT_{i}$. Thus, $[\ind_{T}M: T_{i}]=b_{i}-a_{i}$. By Lemma \ref{L-31},
$$T^{M}_{1}\stackrel{}{\longrightarrow}T^{M}_{0}\stackrel{}{\longrightarrow}M\rightarrow \Sigma T^{M}_{1}$$
is a triangle in $\mathcal{A}$, and
$$[\ind_{T'}M: T_{i}]=\left\{
\begin{aligned}
b_{i}-a_{i}, &  & \text{if}~1\leq i\leq r, \\
0, &  & \text{if}~r+1\leq i\leq n.
\end{aligned}
\right.
$$

$\mathbf{Claim~2.}$ For any $1\leq i\leq r$ and $X\in\mathscr{C}$ with ${\bf dim}\,\mathbb{H}(X)=e$, we have $$\lr{S_{i},e}_{a}=-[e:T_{i}].$$

By \cite[Lemma 2.3]{Pa},
\begin{align*}
\lr{S_{i},e}_{a}&=\lr{S_{i},\mathbb{H}(X)}_{a}=\lr{S_{i},\mathbb{H}(X)}-\lr{\mathbb{H}(X),S_{i}}\\
&=[{\rm coind}_{T'} (X): T_{i}]-[{\rm ind}_{T'}(X): T_{i}]=[{\rm coind}_{T'} (X)-{\rm ind}_{T'}(X): T_{i}]\\
&=[-{\rm ind}_{T'} (\Sigma^{-1}X)-{\rm ind}_{{T'}} (X): T_{i}]=-[{\rm ind}_{T'} (\Sigma^{-1}X)+{\rm ind}_{T'}(X): T_{i}].
\end{align*}
Consider the $\mathbb{E}$-triangle
$$\Sigma^{-1}X\stackrel{}{\longrightarrow}I(\Sigma^{-1}X)\stackrel{}{\longrightarrow}X\dashrightarrow$$
in $\mathscr{C}$ with $I(\Sigma^{-1}X)\in \mathcal{P}$. For each $1\leq i\leq r$, by Claim 1, we get
\begin{align*}
-[e:T_{i}]&=-[{\rm ind}_{T} (\Sigma^{-1}X)-[I(\Sigma^{-1}X)]+{\rm ind}_{T}(X): T_{i}]\\
&=-[{\rm ind}_{T'} (\Sigma^{-1}X)+{\rm ind}_{T'}(X): T_{i}]\\
&=\lr{S_{i},e}_{a}.
\end{align*}
Therefore, by these two Claims we complete the proof.
\end{proof}
\begin{corollary}\label{tuilun4.6}
Assume that $\mathscr{C}$ is a 2-CY Frobenius exact category. Then $$\mathbb{X}^{T}_{M}(x_1,\cdots,x_r,1,\cdots,1)=X'^{T}_{M}(x_1,\cdots,x_r,1,\cdots,1).$$
\end{corollary}
\begin{proof} By Proposition \ref{www} and \cite[Theorem 3.3(b)]{Chan}, $$\mathbb{X}^{T}_{M}(x_1,\cdots,x_r,1,\cdots,1)=X^{T'}_{\Sigma M}=X'^{T}_{M}(x_1,\cdots,x_r,1,\cdots,1).$$
\end{proof}

For each $1\leq i\leq n$, let $S_{i}$ be the top of the projective $B$-module $\mathbb{F}(T_{i})$. For any $M\in\mathscr{C}$, by \cite[Lemma 3.4]{Wl}, there exists an $\mathbb{E}$-triangle
$$T^{M}_{1}\stackrel{}{\longrightarrow}T^{M}_{0}\stackrel{f}{\longrightarrow}M\stackrel{}\dashrightarrow$$
with $T^{M}_{0},T^{M}_{1}\in\add T$ such that $f$ is a minimal right $\add T$-approximation. Thus,
$$\mathbb{F}T^{M}_{1}\stackrel{}{\longrightarrow}\mathbb{F}T^{M}_{0}\stackrel{}{\longrightarrow}\mathbb{F}M{\longrightarrow}0$$
is a minimal projective presentation of $\mathbb{F}M$. Since each differential in the complex
$$0{\longrightarrow}{\rm Hom}_B(\mathbb{F}T^{M}_{0},S_{i})\stackrel{}{\longrightarrow}{\rm Hom}_B(\mathbb{F}T^{M}_{1},S_{i})\stackrel{}{\longrightarrow}\cdots$$
vanishes, ${\rm Hom}_{B}(\mathbb{F}M,S_{i})\cong{\rm Hom}_{B}(\mathbb{F}T^{M}_{0},S_{i})$ and $\Ext^{1}_B(\mathbb{F}M,S_{i})\cong{\rm Hom}_B(\mathbb{F}T^{M}_{1},S_{i})$. Hence,
\begin{flalign*}\lr{\mathbb{F}M,S_{i}}&={\rm dim}_k{\rm Hom}_{B}(\mathbb{F}T^{M}_{0},S_{i})-{\rm dim}_k{\rm Hom}_{B}(\mathbb{F}T^{M}_{1},S_{i})\\&=[\mathbb{F}T^{M}_{0}:\mathbb{F}T_{i}]-[\mathbb{F}T^{M}_{1}:\mathbb{F}T_{i}]
\\&=[\ind_{T} M:T_{i}].\end{flalign*}

Now, we assume that $\mathscr{C}$ is a 2-CY Frobenius exact category. Let us go on comparing the cluster characters (\ref{echaracter}) and (\ref{fucharacter}).

By \cite[Lemma 3.1]{Gorsky}, we obtain
$$\mathbb{G}(\Sigma M)={\rm Hom}_{\mathcal{A}}(T,\Sigma M)\cong\Ext_{\mathscr{C}}^{1}(T,M).$$
Therefore,
\begin{equation}\label{F1}
\mathbb{X}^{T}_{M}=\prod_{i=1}^{n}x_{i}^{\lr{\mathbb{F}M,S_{i}}}\sum_{e}\chi({\rm Gr}_{e}(\Ext_{\mathscr{C}}^{1}(T,M)))\prod_{i=1}^{n}x_{i}^{-[e:T_{i}]}.
\end{equation}

\begin{condition}\label{TT} For any indecomposable object $M\in\mathscr{C}$, if $M\notin\add T$, then $$[e:T_{i}]=\lr{N,S_{i}}_{3}$$ for any $C$-submodule $N$ of $\Ext_{\mathscr{C}}^{1}(T,M)$ with ${\bf dim}\,N=e$ in $K_{0}(\mod C)$ and $r+1\leq i\leq n$.
\end{condition}
\begin{proposition}\label{Fin} Let $\mathscr{C}$ be a 2-CY Frobenius exact category. Then $\mathbb{X}^{T}_{?}=X'^{T}_{?}$ if and only if the condition {\rm (\ref{TT})} holds.
\end{proposition}
\begin{proof}
Comparing (\ref{fucharacter}) and (\ref{F1}), we conclude that the necessity is clear.

Conversely, for any indecomposable object $M\in\mathscr{C}$, $\Ext_{\mathscr{C}}^{1}(T,M)=0$ if and only if $M\in\add T$. Thus, if $M\in \add T$, clearly, $\mathbb{X}^{T}_{M}=X'^{T}_{M}$.

Suppose that $M\notin \add T$. By Corollary \ref{tuilun4.6}, we obtain
$$\mathbb{X}^{T}_{M}(x_1,\cdots,x_r,1,\cdots,1)=X^{T'}_{\Sigma M}=X'^{T}_{M}(x_1,\cdots,x_r,1,\cdots,1).$$
Thus, $[e:T_{i}]=\lr{N,S_{i}}_{3}$ for any $C$-submodule $N$ of $\Ext_{\mathscr{C}}^{1}(T,M)$ with ${\bf dim}\,N=e$ in $K_{0}(\mod C)$ and $1\leq i\leq r$.
Using the condition {\rm (\ref{TT})}, we get $\mathbb{X}^{T}_{M}=X'^{T}_{M}$.
\end{proof}
We finish this section with an example illustrating Proposition \ref{Fin}.
\begin{example}\label{exa} Let $\mathscr{C}$ be the category of finite-dimensional modules over the preprojective algebra of a Dynkin quiver of type $\mathbb{A}_{2}$. Then $\mathscr{C}$ is a 2-CY Frobenius exact category (cf. \cite[Section 4]{Ke}). The Auslander--Reiten quiver of $\mathscr{C}$  is given by:
$$
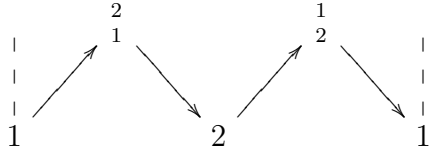
\xymatrix{
    \ar@{--}[d]_{}& 2\atop1 \ar[dr]_{}  &    &   1\atop2 \ar[dr]_{}  &\ar@{--}[d]_{} \\
        1\ar[ur]^{} & &  2\ar[ur]^{} &   &   1 }
$$
The two vertical dashed lines should be identified. There are exactly four indecomposable modules up to isomorphism:
$$\xymatrix{
 M=2,  &  T_{1}=1,& T_{2}={1\atop2}, &  T_{3}={2\atop1}.  }
$$
The module $T=T_{1}\oplus T_{2}\oplus T_{3}$ is a $2$-cluster tilting object in $\mathscr{C}$ and the algebra $B=\End_{\mathscr{C}}(T)$ is isomorphic to the path algebra of the quiver
$$ 
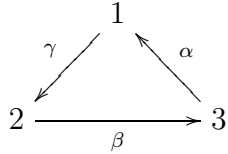
\xymatrix{
                & 1\ar[ld]_-{\gamma}             \\
  2\ar[rr]_-{\beta} & & 3 \ar[lu]_-{\alpha}          }
$$
with zero relations $\alpha\beta$ and $\beta\gamma$. The Auslander--Reiten quiver of $\mod B$ is as follows:
$$
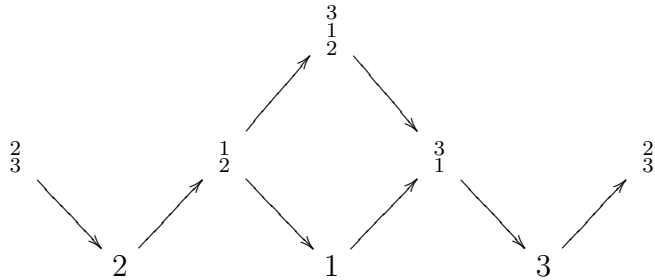
\xymatrix{
   &  &   &   {\begin{smallmatrix}3\\1\\2\end{smallmatrix}}\ar[dr]_{} &  & &   \\
 {\begin{smallmatrix}2\\3\end{smallmatrix}}  \ar[dr]_{}  &  &    {\begin{smallmatrix}1\\2\end{smallmatrix}} \ar[dr]_{} \ar[ur]^{} & &  {\begin{smallmatrix}3\\1\end{smallmatrix}}  \ar[dr]_{}  & &   {\begin{smallmatrix}2\\3\end{smallmatrix}} \\
  &  2\ar[ur]^{} &  & 1 \ar[ur]^{} &   &  3  \ar[ur]^{} &    }
$$
The Cartan matrix of $B$ is $C_B=\begin{pmatrix}1&0&1\\
1&1&1\\
0&1&1\end{pmatrix}$ and the Euler matrix is $$(C_B^{-1})^{\rm tr}=\begin{pmatrix}0&-1&1\\
1&1&-1\\
-1&0&1\end{pmatrix}.$$
Thus, $\lr{S_{1}, S_{2}}_{3}=-1$ and $\lr{ S_{1}, S_{3}}_{3}=1$.

It is easy to see that $$\mathbb{G}(\Sigma M)=\mathbb{G}(T_{1})\cong S_{1}$$
as $B$-modules. It is straightforward to check that  $\mathbb{X}_{T_{i}}^{T}=X'^{T}_{T_{i}}=x_{i}$ for $i=1,2,3$. Since $\mathbb{G}(\Sigma M)$ is a simple module, by Proposition \ref{Fin}, we know that $\mathbb{X}^{T}_{M}=X'^{T}_{M}$ if and only if $[{\bf dim}\,S_{1}:T_{i}]=\lr{S_{1},S_{i}}_{3}$ for $i=2,3$. Since there is an exact sequence
$$0\stackrel{}{\longrightarrow}2\stackrel{}{\longrightarrow}{1\atop2}\stackrel{}{\longrightarrow}1\stackrel{}{\longrightarrow}0,$$
by Proposition \ref{2333}, we get $$\Theta(1)=\ind_{T}(2)-\ind^{\rm op}_{T}(2)=([T_{3}]-[T_{1}])-([T_{2}]-[T_{1}])=[T_{3}]-[T_{2}].$$
Thus,
$[{\bf dim}\,S_{1}:T_{2}]=-1$~{and}~$[{\bf dim}\,S_{1}:T_{3}]=1$. Hence,
$$[{\bf dim}\,S_{1}:T_{2}]=\lr{S_{1}, S_{2}}_{3}\quad\text{and}\quad[{\bf dim}\,S_{1}:T_{3}]=\lr{ S_{1}, S_{3}}_{3}.$$
Therefore, the cluster characters $\mathbb{X}^{T}_{?}$ and $X'^{T}_{?}$ on $\mathscr{C}$ are equal.
\end{example}
\begin{remark}
After the paper was presented on arXiv, Faber--Marsh--Pressland proved the condition {\rm (\ref{TT})} automatically holds (cf. \cite[Proposition 5.4]{FMP}).
That is, for each 2-CY Frobenius exact category $\mathscr{C}$, the cluster character $\mathbb{X}^{T}_{?}$ coincides with Fu--Keller's cluster character $X'^{T}_{?}$.
\end{remark}


\end{document}